\documentclass[11pt]{amsart}
\usepackage{amsmath,amssymb,amsthm,graphicx}
\usepackage[usenames]{color}
\usepackage[shortlabels]{enumitem}
\usepackage{hyperref}

\newtheorem{thm}{Theorem}
\newtheorem{lem}[thm]{Lemma}

\theoremstyle{definition}
\newtheorem{defn}[thm]{Definition}

\newtheorem{rmk}[thm]{Remark}
\newtheorem{exmp}[thm]{Example}

\newcommand{\CPb}{\overline{\mathbb{CP}}{}^{2}}
\newcommand{\CP}{{\mathbb{CP}}{}^{2}}
\newcommand{\CPS}{{\mathbb{CP}}{}^{1}}

\title[Symplectic Surgeries and new Lefschetz fibrations]
{Symplectic Surgeries along Certain Singularities and new Lefschetz fibrations} 
\begin{document}

\author{Anar Akhmedov}
\address{School of Mathematics, 
University of Minnesota, 
Minneapolis, MN, 55455, USA}
\email{akhmedov@math.umn.edu}

\author{Ludmil Katzarkov}
\address{Fakult\"{a}t f\"{u}r Mathematik, Universit\"{a}t Wien, Oskar-Morgenstern-Platz 1, 1090 Wien, Osterreich, and University of Miami, USA, and National Research University Higher School of Economics, Russian Federation, }
\email{lkatzarkov@gmail.com}


\date{September 05, 2018}



\begin{abstract}We define a new 4-dimensional symplectic cut and paste operations arising from the \emph{generalized star relations} $(t_{a_0}t_{a_1}t_{a_2} \cdots t_{a_{2g+1}})^{2g+1} \\ = t_{b_1} t_{b_2}^{g}t_{b_3}$, also known as the \emph{trident relations}, in the mapping class group $\Gamma_{g,3}$ of an orientable surface of genus $g\geq1$ with $3$ boundary components. We also construct new families of Lefschetz fibrations by applying the (generalized) star relations and the \emph{chain relations} to the families of words $(t_{c_1}t_{c_2} \cdots t_{c_{2g-1}}t_{c_{2g}}t_{{c_{2g+1}}}^2t_{c_{2g}}t_{c_{2g-1}} \cdots t_{c_2}t_{c_1})^{2n} = 1$, $(t_{c_1}t_{c_2} \cdots t_{c_{2g}}t_{c_{2g+1}})^{(2g+2)n} = 1$ and $(t_{c_1}t_{c_2} \cdots t_{c_{2g-1}}t_{c_{2g}})^{2(2g+1)n} = 1$ in the mapping class group $\Gamma_{g}$ of the closed orientable surface of genus $g \geq 1$ and $n \geq 1$. Furthemore, we show that the total spaces of some of these Lefschetz fibraions are irreducible exotic symplectic $4$-manifolds. Using the degenerate cases of the generalized star relations, we also realize all elliptic Lefschetz fibrations and genus two Lefschetz fibrations over $\mathbb{S}^{2}$ with non-separating vanishing cycles. \end{abstract}

\maketitle

\section{Introduction} Since the seminal works of \cite{D1, GS} Donaldson and Gompf, establishing a correspondence between Lefschetz pencils and symplectic $4$-manifolds, Lefschetz fibrations have been studied extensively in $4$-dimensional symplectic topology. Given a Lefschetz fibration over $\mathbb{S}^2$, one can associate to it a word in the mapping class group of a fiber composed of right-handed Dehn twists. Conversely, given such a factorization in the mapping class group, one can construct a Lefschetz fibration over $\mathbb{S}^2$ by attaching $2$-handles along the corresponding vanishing cycles. 

Recently, there has been some interest in trying to understand the topological interpretation of various relations in the mapping class group. A particularly well understood case is the generalized lantern relation (also known as the daisy relation in literature), which corresponds to the symplectic operation of rational blowdown \cite{EG, EMVHM}. The reader is also referred to \cite{KS}, where a new symplectic surgery operation along some negative definite tree-shaped configurations of symplectic spheres was given. The first author and his collaborators have used some of these techniques in the following articles \cite{AP, AMon1, AMon2, AMS} to construct various exotic $4$-manifolds with small topology. One would hope to understand the topological meaning of many other well-known relations in the mapping class group, such as the generalized star relations and the generalized chain relations, and construct new symplectic 4-manifolds using the corresponding symplectic operations.

Motivated by the above mentioned results and problems, our goal in this paper is to construct new 4-dimensional symplectic surgery operations arising from the generalized star relations (GSR for short) $(t_{a_0}t_{a_1}t_{a_2} \cdots t_{a_{2g+1}})^{2g+1} \\ = t_{b_1} t_{b_2}^{g}t_{b_3}$ in the mapping class group $\Gamma_{g,3}$ of an orientable surface of genus $g\geq1$ with $3$ boundary components. By applying the sequence of GSR substitutions, chain substitutions, and conjugations to the families of hyperelliptic words $(c_1c_2 \cdots c_{2g-1}c_{2g}{c_{2g+1}}^2c_{2g}c_{2g-1} \cdots c_2c_1)^{2n} = 1$, $(c_1c_2 \cdots c_{2g}c_{2g+1})^{(2g+2)n} = 1$, and $(c_1c_2 \cdots c_{2g-1}c_{2g})^{2(2g+1)n} = 1$ in the mapping class group of the closed orientable surface of genus $g$ for any $g \geq 1$, we also construct families of Lefschetz fibrations. Furthermore, we show that total spaces of some of these Lefschetz fibraions are irreducible exotic symplectic $4$-manifolds. Using the degenerate cases of the generalized star relations, we also realize all elliptic Lefschetz fibrations and genus two Lefschetz fibrations over $\mathbb{S}^2$ with non-separating vanishing cycles. Some additional examples will be given in the upcoming work.

The organization of our paper is as follows. In Sections 2 and 3, we recall the main definitions and the results on mapping class groups and Lefschetz fibrations that will be used throughout the paper. In Section 4, we prove our main theorem, which shows that the GSR surgery can be defined symplectically. In Section 5, we construct new families of Lefschetz fibrations over $\mathbb{S}^2$ by applying GSR substitutions and chain substitutions to the words listed above. We also discuss some known examples,  but from a different point of view and exhibit some new words. The work undertaken in this paper will be extended in future research, where we will apply GSR surgery and further generalization of GSR surgery to constuct interesting symplectic $4$-manifolds. We will also study some applications to the mirror symmetry program and the relation of our work to \cite{AK}.

\section{Mapping Class Groups} 

\begin{defn} Let $\Sigma_{g}^k$ be a compact, connected, and oriented $2$-dimensional surface of genus $g$ with $k$ boundary components. Let $Diff^{+}\left( \Sigma_{g}^k\right)$ denote the group of all orientation-preserving self-diffeomorphisms of $\Sigma_{g}^k$ which are the identity on the boundary and $ Diff_{0}^{+}\left(\Sigma_{g}\right)$ be the subgroup of $Diff^{+}\left(\Sigma_{g}\right)$ consisting of all orientation-preserving self-diffeomorphisms that are isotopic to the identity. The isotopies are assumed to fix the points on the boundary. 
\emph{The mapping class group} $\Gamma_{g}^k$ of $\Sigma_{g}^k$ is defined to be the group of isotopy classes of orientation-preserving diffeomorphisms of $\Sigma_{g}^k$, i.e.,
\[
\Gamma_{g}^k=Diff^{+}\left( \Sigma_{g}^k\right) /Diff_{0}^{+}\left(
\Sigma_{g}^k\right) .
\]
For simplicity, we will use $\Sigma_g = \Sigma_g^0$ and $\Gamma_g = \Gamma_g^0$.
The hyperelliptic mapping class group $H_{g}$ of $\Sigma_{g}$ is defined as the subgroup of $\Gamma_g$ consisting of all isotopy classes which commute with the isotopy class of the hyperelliptic involution $\iota: \Sigma_{g}\rightarrow \Sigma_{g}$. 
\end{defn}

\begin{defn} Let $\alpha$ be a simple closed curve on $\Sigma_{g}^k$. A \emph{right handed} \emph{Dehn twist} about $\alpha$ is a diffeomorphism $t_{\alpha}: \Sigma_{g}^k\rightarrow \Sigma_{g}^k$ obtained by cutting the surface $\Sigma_{g}^k$ along $\alpha$ and gluing the ends back after rotating one of the ends $2\pi$ to the right. 

\end{defn}

It is well-known that the mapping class group $\Gamma_g^n$ is generated by Dehn twists. It is an elementary fact that the conjugate of a Dehn twist is again a Dehn twist. Indeed, if $\phi: \Sigma_{g}^k\rightarrow \Sigma_{g}^k$ is an orientation-preserving diffeomorphism, then it is easy to see that $\phi \circ t_\alpha \circ \phi^{-1} = t_{\phi(\alpha)}$. Also, if $\beta$ and $\gamma$ are two simple closed curves on $\Sigma_{g}^k$ with intersection number $\mu([\beta],[\gamma]) = l \geq 0$, then $t_{\alpha}(\beta) = {\alpha}^{l}\beta$. The following lemma is easy to prove.

\begin{lem}\label{com&braid.lem} Let $\alpha$ and $\beta$ be two simple closed curves on $\Sigma_{g}^k$. If $\alpha$ and $\beta$ are disjoint, then their corresponding Dehn twists satisfy the commutativity relation: $t_{\alpha}t_{\beta}=t_{\beta}t_{\alpha}.$ If $\alpha$ and $\beta$ transversely intersect at a single point, then their corresponding Dehn twists satisfy the braid relation: $t_{\alpha}t_{\beta}t_{\alpha}=t_{\beta}t_{\alpha}t_{\beta}.$
\end{lem}

\subsection{Generalized Star Relations and Generalized Star Substitutions}

In what follows, we will consider a family of relations in the mapping class group which will be necessary for our main purpose. Some of these relations will aid us in defining the symplectic surgeries along certain plumbed $4$-manifolds. We will also apply these relations to known Lefschetz fibrations to construct new Lefschetz fibrations. Let us first recall the well known star relation in the mapping class group $\Gamma_{1}^3$.

\begin{lem}\label{star}{\bf (Star relation)}. Let $c_1$, $c_2$, $c_3$ denote the three boundary curves of surface $\Sigma_{1}^3$, and let $a_0$, $a_1$, $a_2$, $a_3$ be simple closed curves as in Figure~\ref{starcurves}. Then the following relation holds, which is known as the \textit{star relation}, in the mapping class group $\Gamma_{1}^3$
\begin{align}
(t_{a_0}t_{a_1}t_{a_2} t_{a_3})^{3} = t_{c_1}t_{c_2}t_{c_3} \label{eq:1}
\end{align}
\end{lem}

\begin{figure}[ht]
\begin{center}
\includegraphics[scale=.55]{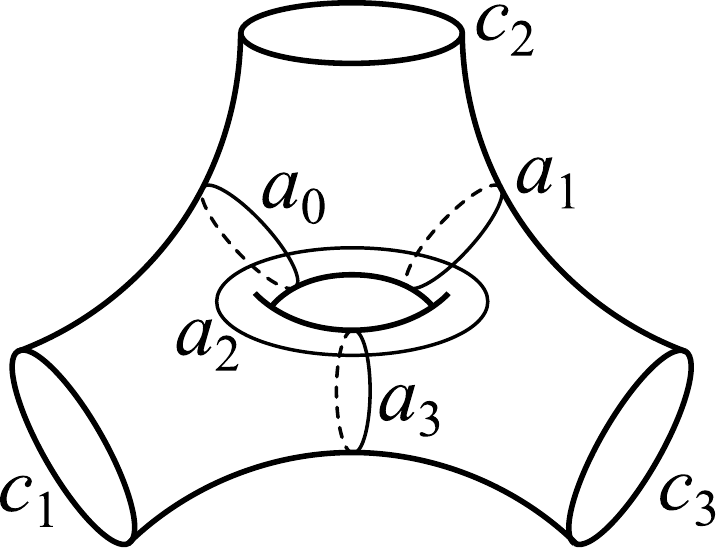}
\caption{The curves defining the star relation}
\label{starcurves}
\end{center}
\end{figure}

A generalization of the star relation considered in the works of Labruere and Paris\cite{LP} and also more recently by Parlak and Stukow in \cite{PS}, which inspired us to to define and explore GSR surgery operation. 
Let us now recall a lemma from \cite{PS} which will be of use to us. For convenience, we include the proof given  \cite{PS} (see also \cite{LP} for another proof of this lemma). The figures are borrowed from the \cite{PS}.

\begin{lem}{\bf (Generalized Star Relation)} Let $c_1$, $c_2$, $c_3$ denote the three boundary curves of a surface $\Sigma_{g}^3$, and let $a_0$, $a_1$, $\cdots$, $a_{2g+1}$ be simple closed curves as in Figure~\ref{curvesGSR}. Then the following relation holds, which is known as the \textit{trident relation} or \textit{a generalized star relation}, in the mapping class group $\Gamma_{g}^3$
\begin{align}
(t_{a_0}t_{a_1}t_{a_2} \cdots t_{a_{2g+1}})^{2g+1} =t_{c_1}{t_{c_2}}^{g}t_{c_3}
\end{align}
\end{lem}

\begin{proof} Let $T = t_{a_0}t_{a_1}t_{a_2} \cdots t_{a_{2g+1}}$ and let $a_{2g+2}$ simple closed curve as in Figure 2. It is easy to see that $T(a_{m}) = a_{m+1}$ for $m = 2,3,\cdots,2g + 1$, and $T(a_{2g+2}) = a_2$. Using these formulas, one can easily verify that 

\begin{figure}[ht]
\begin{center}
\includegraphics[scale=.55]{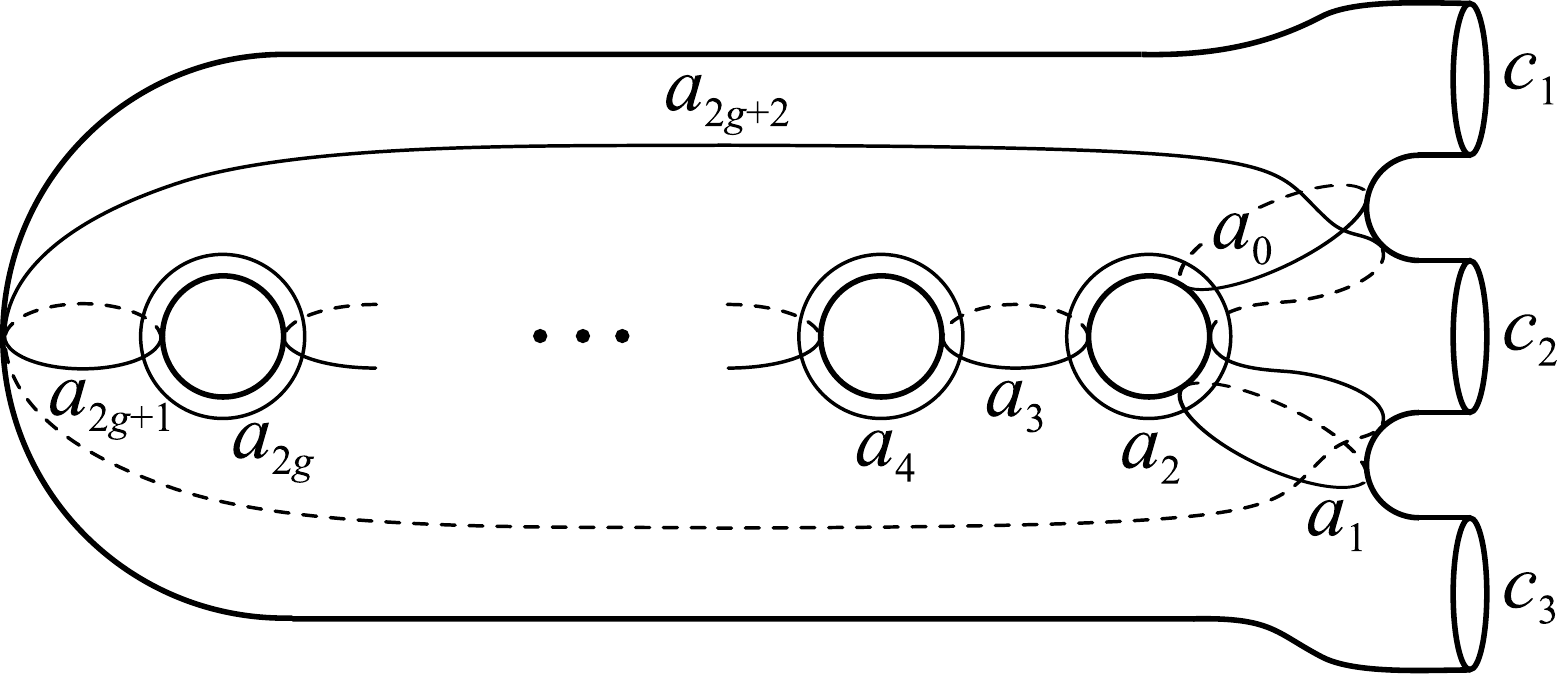}
\caption{The curves defining the GSR relation}
\label{curvesGSR}
\end{center}
\end{figure}

\begin{align}
T^{2g+1}(a_m) =a_m = t_{c_1}{t_{c_2}}^{g}t_{c_3}(a_{m})
\end{align}

for $m=2,3,\cdots,2g+1$.

Let us consider arcs $\lambda_1$, $\lambda_2$, $\cdots$, $\lambda_{2g+2}$ and $\mu_{1}$, $\mu_{2}$, $\cdots$, $\mu_{2g+1}$ as shown in the Figure~\ref{curves1}. Note that $T(\lambda_m)=\lambda_{m+1}$ for $m=1,2,\cdots,2g+1$. In particular, we have

\begin{align}
T^{2g+1}(\lambda_1) =\lambda_{2g+2} = t_{c_1}{t_{c_2}}^{g}t_{c_3}(\lambda_1)
\end{align}

As for the circles $\mu_m$, if $H$ is a half twist about the curve $c_2$ such that $H^{2} = t_{c_2}$, then
$H^{-1}T(\mu_{m}) = \mu_{m+1}$, for $m = 1,2,\cdots,2g$, and $T(\mu_{2g+1}) = \mu_{1}$. Therefore, the following equations holds
\begin{align}
T(H^{-1}T)^{2g}(\mu_{1}) = \mu_{1},\\
T^{2g+1}(\mu_{1}) = H^{2g}(\mu_{1}) = {t_{c_2}}^{g}(\mu_{1}) = t_{c_1}{t_{c_2}}^{g}{t_{c_3}}(\mu_{1}).
\end{align}

\begin{figure}[ht]
\begin{center}
\includegraphics[scale=.55]{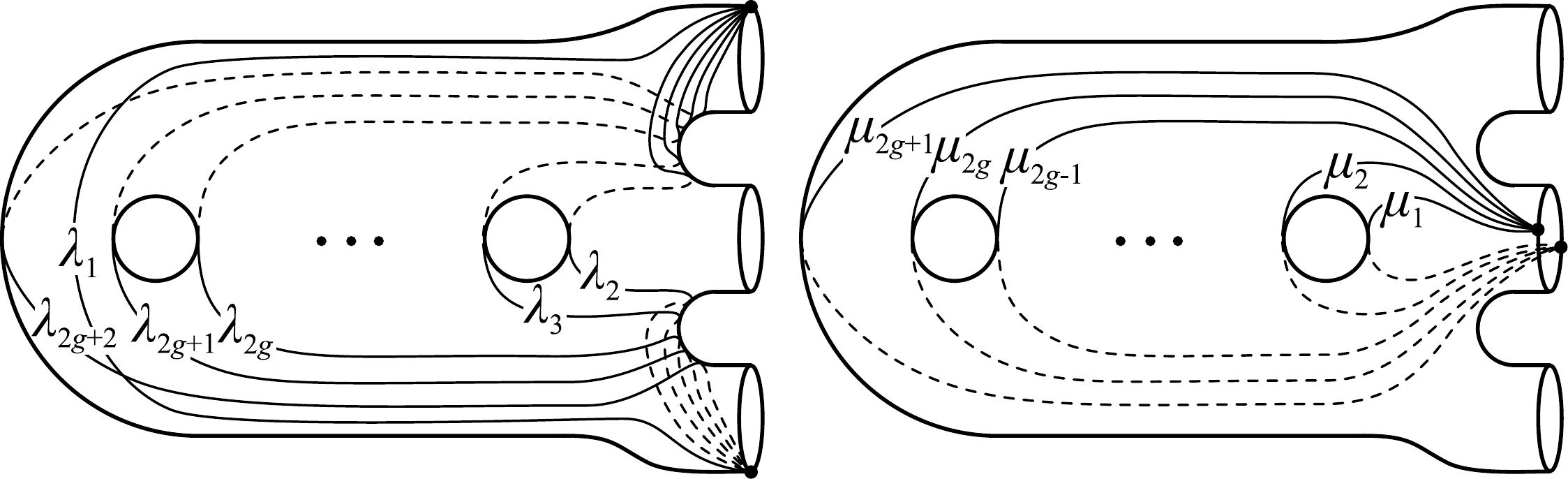}
\caption{Arcs $\lambda_i$ and $\mu_i$ in the proof of the trident relation}
\label{curves1}
\end{center}
\end{figure}

The above equations imply that $T_{2g+1}$ and $t_{c_1}{t_{c_2}}^{g}t_{c_3}$ agree on the set of all simple closed curves and arcs listed. Since this set $\{a_2$, $a_3$, $\cdots$, $a_{2g+1}$,$\lambda_{1}$, $\mu_{1}\}$ satisfies the assumptions of the Alexander Method, we can conclude that $T^{2g+1} =  t_{c_1}{t_{c_2}}^{g}t_{c_3}$ .  
\end{proof}

\begin{rmk}\label{remk1} When $g=1$, the trident relation is known as the \textit{star relation} (see \cite{De}, \cite{Jo}). Also, notice that the generalized star relation generalizes both the star and the chain relations in the mapping class group. This is left as an easy exercise to a reader. \end{rmk}

Let us recall the chain relation in the mapping class group, which will also be used in our computation. The chain relation can also be interpreted as a symplectic surgery operation, using the approach given in our paper (Note the Remark~\ref{remk1}).

\begin{defn} Let $a_1$, $\ldots$, $a_m$ be an ordered $m$-tuple of simple closed curves on an orientable surface $\Sigma_{g}$. It is called \emph{a chain} of lenght $m$ if the curves satisfy the following two conditions:
\begin{enumerate}
\item the curves $a_i$ and $a_{i+1}$ intersect at one point for all $i$ between $1$ and $n-1$,
\item the curves $a_i$ and $a_j$ do not intersect if $|i-j| > 1.$
\end{enumerate}
\end{defn}

The following two lemmas are well known, see for instance \cite{EN}.

\begin{lem}{\bf (Chain relation)}\label{chain} Let $a_1$, $\ldots$, $a_m$, where $m\geq~1$, be a chain of simple closed curves on an orientable surface $\Sigma_{g}$. If $m$ is odd, a regular neighbourhood of a chain $\bigcup\limits_{i=1}^{m}a_i$ is a subsurface of $\Sigma_{g}$ which is of genus $h = (m-1)/2$ and has two boundary components $d_1$ and $d_2$. Then the following relation holds in $\Gamma_{h}^{2}$ 
		\[(t_{a_1}\cdots t_{a_m})^{m} = t_{d_1}t_{d_2}. \]
		If $m$ is even, the boundary of a regular neighbourhood of $\bigcup\limits_{i=1}^{m}a_i$ is a subsurface of $\Sigma_{g}$ which is of genus $h = m/2$ and has one boundary components $c$, then in $\Gamma_{h}^{1}$
		\[ (t_{a_1}\cdots t_{a_m})^{2m+2} = t_{c} \] 
\end{lem}

The next lemma can be proved using an induction.

\begin{lem}\label{power} Let $a_1, \cdots, a_m$ be a chain of lenght $m$ on $\Sigma_{g}$. For any $k=1, \cdots, m-1$, the following relation holds up to the braid equivalence: 
\begin{align*}
{(t_{c_1}t_{c_2} \cdots t_{c_m})}^{k+1} = {(t_{c_1}t_{c_2}\cdots t_{c_k})}^{k+1} (t_{c_{k+1}}t_{c_k}\cdots t_{c_2} t_{c_1})\\(t_{c_{k+2}}t_{c_{k+1}} \cdots t_{c_3} t_{c_2}) \cdots (t_{c_m}t_{c_{m-1}} \cdots t_{c_{m-k+1}} t_{c_{m-k}})
\end{align*}
\end{lem}

\begin{defn}\label{GSR substitution}\rm
Let $d_1,\ldots,d_m$ and $e_1,\ldots, e_k$ be simple closed curves on $\Gamma_g^k$, and let $W$ be a product $W=t_{d_1}t_{d_2}\cdots t_{d_l}t_{e_m}^{-1}\cdots t_{e_2}^{-1}t_{e_1}^{-1}$. Suppose that $R=1$ in $\Gamma_g^k$. Let $\varrho$ be a word in $\Gamma_g^k$ including $t_{d_1}t_{d_2}\cdots t_{d_l}$ as a subword: 
\begin{align*}
\varrho=U\cdot t_{d_1}t_{d_2}\cdots t_{d_l} \cdot V, 
\end{align*}
where $U$ and $V$ are words. Thus, we obtain a new word in $\Gamma_g^n$, denoted by $\varrho^\prime$, as follows:
\begin{align*}
\varrho^\prime:&=U\cdot t_{e_1}t_{e_2}\cdots t_{e_m} \cdot V.
\end{align*}

\begin{figure}[ht]
\begin{center}
\includegraphics[scale=.83]{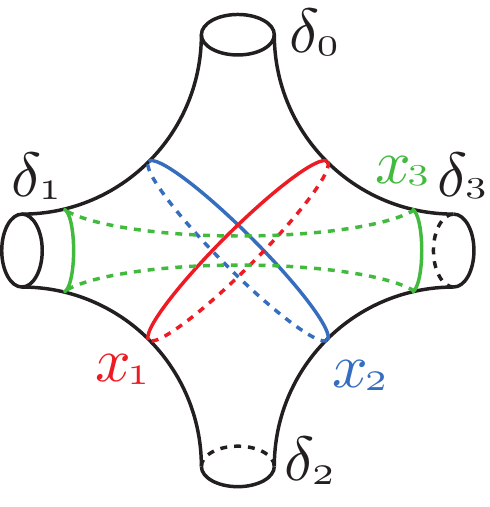}
\caption{The curves defining the lantern relation}
\label{fig:hyper}
\end{center}
\end{figure}

Then, we say that $\varrho^{\prime}$ is obtained by applying a $R$-\textit{substitution} to $\varrho$. In particular, if $R$ is a generalized star relator on genus $g$ surface with $3$ boundary components, then we say that $\varrho^{\prime}$ is obtained by applying a \textit{GSR substitution of type $g$} to $\varrho$. The substitution changing $\varrho^{\prime}$ to $\varrho$ will be called \textit{the reverse GSR substitution of type $g$}. Also, if $R$ is a lantern relator (or, respectively, 
a chain relator), then we say that $\varrho^{\prime}$ is obtained by applying a \textit{lanter substitution} (or, respectively, a \textit{chain substitution}) to $\varrho$. Refer to the Figure~\ref{fig:hyper} for the lantern relation and Lemma~\ref{chain} for the chain relation. 
\end{defn}

\section{Lefschetz fibrations}

\begin{defn}\label{LF}\rm
Let $Y$ be a closed, oriented smooth $4$-manifold. A smooth map $f : Y \rightarrow \mathbb{S}^2$ is a genus $g$ \textit{Lefschetz fibration} if it satisfies the following condition: \\
(a) $f$ is a smooth $\Sigma_g$-bundle away from finitely many critical values $p_1,\ldots, p_m \in \mathbb{S}^2$,\\
(b) for each critical value $p_i$ $(i=1,\ldots,m)$, there exists a unique critical point $q_i$ in the \textit{singular fiber} $f^{-1}(p_i)$ such that about each point $q_i$ and $p_i$ there are local complex coordinate charts agreeing with the orientations of $Y$ and $\mathbb{S}^2$ on which $f$ is given by $f(z_{1},z_{2})=z_{1}^{2}+z_{2}^{2}$, \\
(c) $f$ is relatively minimal, which means no fiber contains a sphere of self-intersection $-1$.
\end{defn}

Each singular fiber of Lefschetz fibration is obtained by collapsing a simple closed curve, the \textit{vanishing cycle}, in the regular fiber. If the curve is a non-separating, then the singular fiber is called  \textit{non-separating}, otherwise it is called  \textit{separating}. The monodromy of the Lefschetz fibration around a singular fiber is given by a right handed Dehn twist along the corresponding vanishing cycle. A Lefschetz fibration $f : Y \rightarrow \mathbb{S}^2$ is \textit{holomorphic} if there are complex structures on both $Y$ and $\mathbb{S}^2$ with holomorphic map $f$. For a genus $g$ Lefschetz fibration over $\mathbb{S}^2$, the product of right handed Dehn twists $t_{\alpha_i}$ about the vanishing cycles $\alpha_i$, for $i = 1,\ldots, m$, gives us the global monodromy of the Lefschetz fibration, the relation $t_{\alpha_1} t_{\alpha_2} \cdots t_{\alpha_m}=1$ in $\Gamma_g$. Such a relation is called the \textit{positive relator}. Conversely, given a positive relator in $\Gamma_g$, one can construct a genus $g$ Lefschetz fibration over $\mathbb{S}^2$ with the vanishing cycles $\alpha_1,\ldots, \alpha_m$. 

By theorems of Kas \cite{Kas} and Matsumoto \cite{Ma}, if $g\geq 2$, there exists a one-to-one correspondence between the isomorphism classes of Lefschetz fibrations and the equivalence classes of positive relators modulo simultaneous conjugations
\begin{align*}
t_{\alpha_1} t_{\alpha_2} \cdots  t_{\alpha_m} \sim t_{\phi(\alpha_1)} t_{\phi(\alpha_2)} \cdots t_{\phi(\alpha_m)} \ \ {\rm for \ all} \ \phi \in \Gamma_g
\end{align*}
and elementary transformations 
\begin{align*}
&t_{\alpha_1} \cdots t_{\alpha_{i-1}} t_{\alpha_i} t_{\alpha_{i+1}} t_{\alpha_{i+2}} \cdots t_{\alpha_m}& &\sim& &t_{\alpha_1} \cdots t_{\alpha_{i-1}} t_{t_{\alpha_i}(\alpha_{i+1})} t_{\alpha_i} t_{\alpha_{i+2}} \cdots t_{\alpha_m},&\\
&t_{\alpha_1} \cdots t_{\alpha_{i-2}} t_{\alpha_{i-1}} t_{\alpha_i} t_{\alpha_{i+1}} \cdots t_{\alpha_m}& &\sim& &t_{\alpha_1} \cdots t_{\alpha_{i-2}} t_{\alpha_i} t_{t_{\alpha_i}^{-1}(\alpha_{i-1})} t_{\alpha_{i+1}} \cdots  t_{\alpha_m}.&
\end{align*}
Let us recall that $\phi t_{\alpha_i}\phi^{-1}=t_{\phi(\alpha_i)}$. We will denote the isomorphism class of a Lefschetz fibration associated to a positive relator $\varrho \ \in \Gamma_g$ by $f_\varrho$. 

For a Lefschetz fibration $f:Y\rightarrow \mathbb{S}^2$, a map $\tau:\mathbb{S}^2\rightarrow Y$ is called a \textit{section} of $f$ if $f\circ \tau={\rm id}_{\mathbb{S}^2}$. The self-intersection of the section $\tau$ defined to be the self-intersection of the homology class $[\tau(\mathbb{S}^2)]$ in $H_2(Y;\mathbb{Z})$. 
Let $\delta_1,\delta_2,\ldots,\delta_k$ be $k$ boundary components of $\Sigma_g^k$. 
If there exists a lift of a positive relator $\varrho = t_{\alpha_1} t_{\alpha_2} \cdots t_{\alpha_m} = 1$ in $\Gamma_g$ to $\Gamma_g^k$ given by
\begin{align}\label{sections}
t_{\widetilde{\alpha}_1} t_{\widetilde{\alpha}_2} \cdots t_{\widetilde{\alpha}_m} = t_{\delta_1}^{n_1} t_{\delta_2}^{n_2} \cdots t_{\delta_k}^{n_k}, 
\end{align}
then $f_\varrho$ admits $k$ disjoint sections and the self-intersection of the $j$-th section is $-n_k$. Here, $t_{\widetilde{\alpha}_i}$ is a Dehn twist mapped to $t_{\alpha_i}$ under $\Gamma_g^k \to \Gamma_g$. Conversely, if a genus-$g$ Lefschetz fibration admits $k$ disjoint sections with self-intersections $-n_{1}$, $\cdots$, $-n_{k}$, then there is such a relation in $\Gamma_g^k$. 

For our purposes, let us also recall the signature formula of Matsumoto and Endo for hyperelliptic Lefschetz fibrations. We use this formula below and also in Section \ref{words}, to compute the signature of the Lefschetz fibrations obtained by the (generalized) star substitutions.

\begin{thm}[\cite{Ma1},\cite{Ma},\cite{E}]\label{sign} Let $f:Y\rightarrow \mathbb{S}^2$ be a genus $g$ hyperelliptic Lefschetz fibration. Let $s_0$ and $s=\Sigma_{h=1}^{[g/2]}s_h$ be the number of non-separating and separating vanishing cycles of $f$, where $s_h$ denotes the number of separating vanishing cycles which separate the surface of genus $g$ into two surfaces, one of which has genus $h$. Then, we have the following formula for the signature 
\begin{eqnarray*}
\sigma(Y)=-\frac{g+1}{2g+1}s_0+\sum_{h=1}^{[\frac{g}{2}]}\left(\frac{4h(g-h)}{2g+1}-1\right)s_{h}.
\end{eqnarray*}
\end{thm}

Let us now apply the formula~\ref{sections} given above to the study the Lefschetz fibrations arising from the star and the generalized star relations of genus $1$ and genus $g \geq 2$ respectively, and determine their sections. This will be helpful in Section~\ref{GSR} in our study of certain Stein fillings determined by these Lefschetz fibrations.

\begin{exmp}\label{Ex13} In this example, we consider the generalized star relations $(t_{a_0}t_{a_1}t_{a_2} \cdots t_{a_{2g+1}})^{2g+1} = t_{b_1} t_{b_2}^{g}t_{b_3}$ in the mapping class group $\Gamma_{g}^{3}$ and lifting this relation to $(t_{a_0}t_{a_1}t_{a_2} \cdots t_{a_{2g+1}})^{2g+1} = 1$ in $\Gamma_g$, we obtain hyperelliptic genus $g$ Lefschetz fibration over $\mathbb{S}^2$ with three disjoint sections, say $\mathcal{S}_{1}$, $\mathcal{S}_{2}$, and $\mathcal{S}_3$, with self-intersections $-1$, $-g$, and $-1$, respectively. Notice that when $g=1$, i.e., in star relation case, the corresponding fibration  with monodromy $(abaa)^{3} = 1$ admits three disjoint $-1$ sections. When $g=1$, the total space of this fibration is $\CP\#9\CPb$. For an arbitrary $g$, the total space of the fibration given by $(t_{a_0}t_{a_1}t_{a_2} \cdots t_{a_{2g+1}})^{2g+1} = 1$ can be identified with $C(g)$ given as in \ref{families}. The Euler characteristic and the signature of
the fibration are given as follows: $(e, \sigma) = (4g^{2} + 2g + 6, -2g^2 -4g -2)$. By blowing down two $-1$ sphere sections $\mathcal{S}_{1}$ and $\mathcal{S}_{3}$ for $g \geq 2$, we obtain genus $g$ Lefschetz pencil with two base points on complex surface $W(g)$ with the invariants $(e, \sigma) = (4g^{2} + 2g + 4, -2g^2 -4g)$. When $g=1$, blowing down all three sections  $\mathcal{S}_{1}$, $\mathcal{S}_{2}$, and $\mathcal{S}_3$, determines genus one Lefschetz pencil on $\CP\#6\CPb$ with three base points. The discussion here will be used below and also in Section~\ref{GSR}, in our construction of certain Stein fillings.
\end{exmp}

\begin{lem}
Let $\rho$ and $\rho'$ be positive relators in $\Gamma_g$ and let $f_{\rho}: Y_\rho \rightarrow \mathbb{S}^2$ and $f_{\rho'}: Y_{\rho'} \rightarrow \mathbb{S}^2$ be the corresponding Lefschetz fibrations, respectively. Suppose that the relator $\rho'$ is obtained by applying a star-substitution to the relator $\rho$. Then Euler characteristic and the signature of a Lefschetz fibration $f_{\rho'}: Y_{\rho} \rightarrow \mathbb{S}^2$ with monodromy $\rho'$ computed as follows: $e(Y_{\rho'})=e(Y_{\rho})-9$, $\sigma(Y_{\rho'})=\sigma(Y_{\rho})+5$
\end{lem}

\begin{proof} Using the number of singular fibers of $f_{\rho}: Y_\rho \rightarrow \mathbb{S}^2$ and $f_{\rho'}: Y_{\rho'} \rightarrow \mathbb{S}^2$, we can compute the Euler characteristic of $Y_\rho$ and $Y_{\rho'}$. Since their number of singular fibers differ by $9$, we obtain the formula for Euler characteristics. The formula for signature follows from Proposition 3.13 given in \cite{EN}, where it was shown that the signature of a star relation is $+5$. \end{proof}

\begin{lem} Let $\rho$ and $\rho'$ be positive relators in $\Gamma_g$ and let $f_{\rho}: Y_\rho \rightarrow \mathbb{S}^2$ and $f_{\rho'}: Y_{\rho'} \rightarrow \mathbb{S}^2$ be the corresponding Lefschetz fibrations, respectively. Suppose that the relator $\rho'$ is obtained by applying a GSR-substitution of type $g$ to the relator $\rho$. Then Euler characteristic and the signature of a Lefschetz fibration $f_{\rho'}: Y_{\rho} \rightarrow \mathbb{S}^2$ with monodromy $\rho$ are computed as follows: $e(Y_{\rho'})=e(Y_{\rho})-4g^{2}-5g$, $\sigma(Y_{\rho'})=\sigma(Y_{\rho})+2g^{2} + 3g$
\end{lem}

\begin{proof} Using the number of singular fibers of $f_{\rho}: Y_\rho \rightarrow \mathbb{S}^2$ and $f_{\rho'}: Y_{\rho} \rightarrow \mathbb{S}^2$, we see that the difference between them is $(2g+2)(2g+1) - (g+2) = 4g^2 +5g$. Consequently, we have the first formula. The formula for the signature follows from the Propositions 3.12 and 3.13 given in \cite{EN}. Notice that when we apply the generalized star relation of type $g$, the left hand side of the equation contributes $-\frac{g+1}{2g+1}(2g+2)(2g+1)$ and the right hand side contributes $(g+2)$ to the signature. Thus, the star substitution contributes $(g+1)(2g+2) - (g+2)$ = $2g^2 + 3g$ to the signature. This verifies the second formula. \end{proof}

\subsection{Some familes of hyperelliptic genus $g$ Lefschetz fibrations}\label{families} In what follows, we introduce some familes of hyperelliptic Lefschetz fibrations, which will serve as the starting building blocks in our construction of new Lefschetz fibrations. More precisely, we will be applying the (generalized) star substitutions relations to these Lefschetz fibrations to construct new Lefschetz fibrations. Let $c_1$, $c_2$, $\cdots$ , $c_{2g}$, $c_{2g+1}$ denote the collection of simple closed curves given in Figure~\ref{fig:hyper1}, and $c_{i}$ denote the right handed Dehn twists $t_{c_i}$ along the curve $c_i$. It is well-known that the following relations hold in the mapping class group $\Gamma_g$:  

\begin{equation}
\begin{array}{l}
A(g) = (c_1c_2 \cdots c_{2g-1}c_{2g}{c_{2g+1}}^2c_{2g}c_{2g-1} \cdots c_2c_1)^2 = 1,  \\
B(g) = (c_1c_2 \cdots c_{2g}c_{2g+1})^{2g+2} = 1,  \\ 
C(g) = ({c_1}^{2}c_2 \cdots c_{2g})^{2g+1} = 1, \\
D(g) = (c_1c_2 \cdots c_{2g-1}c_{2g})^{2(2g+1)} = 1. 
\end{array}
\end{equation}
  
\begin{figure}[ht]
\begin{center}
\includegraphics[scale=.35]{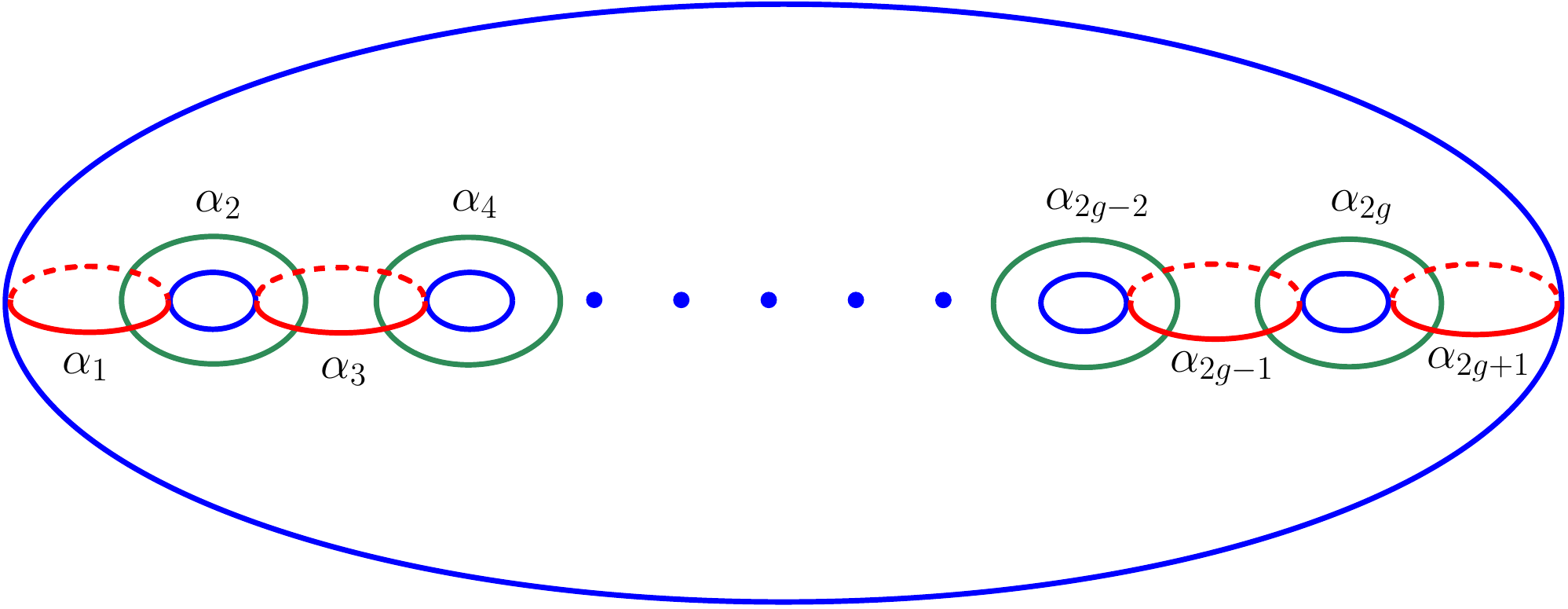}
\caption{Vanishing Cycles of the Genus $g$ Lefschetz Fibration on $X(g)$, $Y(g)$, and $Z(g)$}
\label{fig:hyper1}
\end{center}
\end{figure}

It is known that the Lefschetz fibration defined by the words $B(g)=1$ and $C(g)=1$ have the same total spaces, even though the local monodromies differ from each other. Taking this into consideration, let $X(g)$, $Y(g)$ and $Z(g)$ denote the total spaces of the above genus $g$ hyperelliptic Lefschetz fibrations given by the monodromies $A(g) = 1$, $B(g) = C(g) = 1$, and $D(g) = 1$ respectively, in the mapping class group $\Gamma_g$. 
Notice that when $g=1$, all four Lefschetz fibrations are isomorphic and have the total space $\CP\#9\CPb$. Their monodromies can be converted to each other using the braid relations $c_{1}c_{2}c_{1} = c_{1}c_{2}c_{1}$ in the mapping class group $\Gamma_{1}$. For the monodromy relation $A(g) = 1$, the corresponding genus $g$ Lefschetz fibrations over $\mathbb{S}^2$ has total space $X(g) = \CP\#(4g+5)\CPb$, the complex projective plane blown up at $4g+5$ points. For the relations $B(g) = C(g) = 1$ and $D(g) = 1$, the total spaces of the corresponding genus $g$ Lefschetz fibrations over $\mathbb{S}^2$ are also well-known families of complex surfaces. For example, $Y(2) = K3\#2\CPb$ and $Z(2)$ = \emph{Horikawa surface}, respectively. 

Let us also briefly discuss the branched-cover description of the complex surfaces $Y(g)$ and $Z(g)$. We will use these descriptions in the proofs of our main results. 

The description of $X(g)$ as branched-cover is well-known and can be found in \cite{GS} (see Remark 7.3.5, p.257).      

\begin{lem}\label{E} The genus $g$ Lefschetz fibration on $Y(g)$ over $\mathbb{S}^2$ with the monodromy $(c_1c_2 \cdots c_{2g+1})^{2g+2} = 1$ can be obtained as the double branched covering of $\CP\#\CPb$ branched along a smooth algebraic curve $B$ in the linear system $|2(g+1)\tilde{L}|$, where $\tilde{L}$ is the proper transform of line $L$ in $\CP$ avoiding the blown-up point. Furthermore, this Lefschetz fibration admits two disjoint $-1$ sphere sections.   
\end{lem}

For the proof of the above lemma we refer (\cite{AMon1}, Lemma 10). The proof of the following lemma can be extracted from \cite{GS} [Ex 7.3.27, page 268]. 

\begin{lem}\label{E1} The double branched cover $W(g)$ of $\CP$ along a smooth algebraic curve $B$ in the linear system $|2(g+1)\tilde{L}|$ can be decomposed as the fiber sum of two copies of $\CP\#(g+1)^2\CPb$ along a a complex curve of genus equal $g(g-1)/2$. Moreover, $W(g)$ admits a genus $g$ Lefschetz pencil with two base points, and $Y(g) =  W(g)\#2\CPb$. 
\end{lem}

Let $m$ be any nonnegative integer, and $\mathbb{F}_{m}$ denote $m$-th Hirzebruch surface. $\mathbb{F}_{m}$ admits the structure of holomorphic $\CPS$ bundle over $\CPS$ with the holomorphic sections $\Delta_{+m}$ and $\Delta_{-m}$ with $\Delta_{\pm m} = \pm m$.    

\begin{lem}\label{E3} The genus $g$ Lefschetz fibration on $Z(g)$ over $\mathbb{S}^2$ with the monodromy $(c_1c_2 \cdots c_{2g})^{2(2g+1)} = 1$ can be obtained as the $2$-fold cover of $\mathbb{F}_{2}$ branched over the disjoint union of a smooth curve $C$ in the linear system $|(2g+1)\Delta_{+2}|$ and $\Delta_{-2}$  
\end{lem}

\begin{proof} The Lefschetz fibration on $Z(g) \rightarrow  \CPS$ obtained by composing the branched cover map $Z(g) \rightarrow \mathbb{F}_{2}$ with the bundle map $\mathbb{F}_{2} \rightarrow \CPS$. Since a generic fiber is the double cover of a sphere fiber of $\mathbb{F}_{2}$ branched over $2g + 2$ points, it is a genus $g$ surface. The monodromy of this Lefschetz fibration can be computed from the braid monodromy of the branch curve $C \cup \Delta_{-2}$. This fibration admits a holomorphic sphere section $S$ with $S^2 = -1$, which is obtained by lifting $\Delta_{-2}$ to $Z(g)$.
\end{proof}

\section{Symplectic Surgeries Corresponding to the Generalized Star Relations}\label{GSR}

\subsection{Motivation and some preliminaries} Our generalized star relation (GSR) surgery operation is motivated largely by the rational blowdown construction given in \cite{FS1}. Let us recall that the rational blowdown construction replaces a tubular neighborhood of a certain configuration of negative spheres, denoted by $C_p$, by a rational-homology $4$-ball $\mathbb{B}_{p}$. It it has been used to construct many interesting smooth and symplectic $4$-manifolds. Rational blow-downs operation was shown to be symplectic by Symington \cite{Sym1, Sym2}. Symington showed that both the negative definite plumbing $C_{p}$ and the rational homology ball $\mathbb{B}_{p}$ support symplectic structures with convex boundary inducing the same contact structures. One can try to use the same approach to define other symplectic cut-and-paste operation, by replacing a neighborhood of negative surfaces in symplectic $4$-manifolds. Another motivation for us comes from the paper \cite{EPS}, where the resolution of certain types of singularities were studied using algebraic geometric considerations. The resolution of triple point singularities in \cite{EPS} corresponds to the star surgery operation (i.e., a special case of GSR surgery for $g=1$) discussed in our paper. We plan to study the connection between algebro-geometric and symplectic versions of these operations, understand and utilize some of the examples coming from algebraic geometry in our future work.

Let us now introduce some definitions and results that will be useful in the sequel. The reader is referred to \cite{Sym1, Sym2, Sym3, gompf} for the background and more details.

\begin{defn} Let $\mathcal{C} = C_1 \cup \cdots \cup C_{n} \subset (X, \omega)$ be a collection of immersed symplectic surfaces in a symplectic $4$-manifold $(X,\omega)$ intersecting each other $\omega$-orthogonally. Let $\widehat{X}$ be the symplectic $4$-manifold with boundary associated to $X \setminus \mathcal{C}$ . A closed symplectic manifold $(\tilde{X}, \tilde{\omega})$ is called \emph{a generalized symplectic connected sum} of $X$ along $\mathcal{C}$ if there exists a symplectic embedding $\psi: X\setminus \mathcal{C} \rightarrow \tilde{X}$ which extends to a surjective symplectic immersion $\widehat{\psi}: \widehat{X} \rightarrow \tilde{X}$.
\end{defn}

If the above collection of surfaces consist of a pair disjoint symplectomorphic surfaces, then the generalized symplectic connected sum is just a symplectic connected sum that was studied in \cite{gompf}. In \cite{Sym2, Sym3}, Symington, and McDuff and Symington, provided criteria for constructing certain generalized connect normal sums, called $3$-fold and $4$-fold sums respectively. Let us recall one result of \cite{Sym2, Sym3} which we will use.

\begin{thm}\label{3-fold sum} Let $(M, \omega)$ be a symplecitc $4$-manifold. Let $\mathcal{C} = {\{S_{i}; T_{i} \}}_{i=1}^{3}$ be a collection of symplectic surfaces in $M$ such that $S_i$ and $T_i$ are disjoint from both $S_j$ and $T_j$ for $i\neq j$, and $S_i$ intersects $T_i$ $\omega$-orthogonally once. Assume that $[T_{i}]^2 + [S_{i+1}]^2 = -1$ for each $i$, and that $T_i$ is symplectomorphic to $S_{i+1}$. The result of identifying a punctured neighbourhood of $T_i$ with a punctured neighbourhood of $S_{i+1}$ is a generalized symplectic sum, which is called a $3$-fold sum.
\end{thm}

Let us also recall a result from \cite{GaS} to define the GSR surgery.

\begin{thm}\label{convex} If $\mathcal{C} = C_1 \cup \cdots \cup C_{n} \subset (X, \omega)$ is a collection of symplectic surfaces in a symplectic 4-manifold $(X,\omega)$ intersecting each other $\omega$-orthogonally according to the negative definite plumbing graph $G$ and $\nu \mathcal{C} \subset X$ is an open set containing $\mathcal{C}$, then $\mathcal{C}$ admits an $\omega$ convex neighborhood $U_{\mathcal{C}} \subset \nu \mathcal{C} \subset (X, \omega)$. In particular, the complement $X - int(U_{\mathcal{C}})$ is a strong concave filling of its contact boundary.
\end{thm}

Let $\mathcal{C}$ be the union of symplectic surfaces given as in the statement of the above theorem with the plumbing graph $G$. By the assumption, the plumbing graph $G$ is negative definite. Let us denote by $v_j$ and $e_j < 0$ the vertex corresponding to the surface $C_{j}$ and the self-intersection of this surface, respectively. By a fundamental result of Grauert \cite{Ga}, for any such graph $G$ there is a normal surface singularity $(W, 0)$ with resolution dual graph equal to $G$. In our case, the normal surface singularities will be determined by the left hand side of the generalized star relation $(t_{a_0}t_{a_1}t_{a_2} \cdots t_{a_{2g+1}})^{2g+1} = t_{b_1} t_{b_2}^{g}t_{b_3}$.

\subsection{The generalized star relation surgery} Let us now introduce the symplectic surgery operation, which corresponds to the generalized star relation $(t_{a_0}t_{a_1}t_{a_2} \cdots t_{a_{2g+1}})^{2g+1} = t_{c_1} {t_{c_2}}^{g}t_{c_3}$. Let $S_{g}$ be the Stein $4$-manifold determined by the Lefschetz fibration $(t_{a_0}t_{a_1}t_{a_2} \cdots t_{a_{2g+1}})^{2g+1}$ in $\Gamma_g^{3}$. The fibers of the Lefschetz fibration are genus $g$ surfaces with three boundary components, and the base of the fibration is a disk. The Stein filling $S_g$ is simply connected and it has Euler characteristic $4g^{2}+5g$.

Let $\mathbb{P}_{g}$ be a symplectic filling determined by the monodromy $t_{c_1} {t_{c_2}}^{g}t_{c_3}$. In fact, $\mathbb{P}_{g}$ a plumbed smooth $4$-manifold obtained by plumbing a disk bundle over the genus $g$ surface and $(g-1)$ disk bundles over the $2$-spheres, which can be seen from the monodromy $t_{c_1} {t_{c_2}}^{g}t_{c_3}$. In particular, the above theorem applies to the plumbing $\mathbb{P}_{g}$. It is easy to see that the induced open book decomposition on the boundary of $S_{g}$ agrees with the one on the boundary of the symplectic plumbing given $\mathbb{P}_{g}$. Our GSR surgery operation will replace the Stein filling $S_{g}$ determined by the Lefschetz fibration $(t_{a_0}t_{a_1}t_{a_2} \cdots t_{a_{2g+1}})^{2g+1}$ with the symplectic filling $\mathbb{P}_{g}$ determined by the monodromy $t_{c_1} {t_{c_2}}^{g}t_{c_3}$. When $g=1$, i.e. for star relation $(t_{a_0}t_{a_1}t_{a_2} t_{a_{3}})^{3} = t_{c_1} t_{c_2}t_{c_3}$, this operation on symplectic $4$-manifold corresponds to the symplectic connected sum operation. The details are provided bellow.

Let $g \geq  1$ and $\mathbb{P}_g$ be the smooth $4$-manifold obtained by linear plumbing of one disk bundle over the genus $g$ surface and $(g-1)$ disk bundles over the $2$-spheres with Euler numbers given as follows: the first vertex $u_{0}$ of the linear diagram represents a disk bundle over genus $g$ surface with Euler number $-3$, and all other remaining vertices $u_{i}$ (for $1 \leq i \leq g-1$) represents a disk bundle over $2$-sphere with the Euler number $-2$. The boundary of $\mathbb{P}_g$ is the Seifert fibered $3$-manifold which also bounds the Stein filling $S_g$. 


Now assume that there is a symplectic embedding of a Stein filling $S_g$ into a closed symplectic $4$-manifold $X$. Then \emph{the generalized star surgery} manifold $X_g$ is obtained by replacing the Stein filling $S_g$ with the symplectic filling $\mathbb{P}_g$, i.e., $X_g = (X \setminus {S_g)} \cup \mathbb{P}_g$. We can also define a \emph{reverse generalized star surgery} operation, which replaces the symplectic filling $\mathbb{P}_g$ in a closed symplectic $4$-manifold $Y$ with the Stein filling $S_g$, i.e., $Y_g = (X \setminus \mathbb{P}_g) \cup S_g$.

\begin{lem}\label{thm:rb} $e(X_g) = e(X) - 4g^{2}-5g$, $\sigma(X_g) = \sigma(X) +2g^{2}+3g$, ${c_1}^{2}(X_g) = {c_1}^2(X) - g(2g+1)$, and $\chi(X_g) = \chi(X)- \frac{g(g+1)}{2}$.
\end{lem}

\begin{proof} Using the surgery description of $X_g$ and the Euler characteristics of $S_g$ and $\mathbb{P}_g$, we compute $e(X_g) = e(X) - e(S_g)+ e(\mathbb{P}_g) =  e(X) - (2g+2)(2g+1) - (g+2) = 4g^2 +5g$. Since both $4$-manifolds $S_g$ and $\mathbb{P}_g$ have negative definite intersection forms of rank $(g+1)(2g+1)$ and $(g+2)$ respectively, we compute $\sigma(X_{g}) = \sigma(X) + (g+1)(2g+2) - (g+2) = \sigma(X) + 2g^2 + 3g$. Consequently, we have $c_{1}^{2}(X_{g}):=2e(X_{g}) + 3\sigma(X_{g}) =  {c_1}^2(X) - g(2g+1)$ and $\chi(X_{g}):= (e(X_{g} + \sigma(X_{g})/4 = \chi(X)- g(g+1)/2$. 

\end{proof}

\begin{thm} The GSR surgery and the reverse GSR surgery operations are both symplectic operations.
\end{thm}

\begin{proof} We will consider the cases when $g=1$, and $g \geq 2$ separately. The treatment of the case $g=1$ provides a good intuition for the general case of $g \geq 2$.

\emph{Case 1:} The reverse star surgery along the plumbing $\mathbb{P}_1$ on symplectic $4$-manifold $Y$ corresponds to the symplectic connected sum of $Y$ with $\CP\#6\CPb$ along a torus of square $-3$ in $X$ and a symplectic torus of square $+3$ representing the class $3H - e_1- \cdots - e_6$ in $\CP\#6\CPb$. The complement of torus $3H - e_1 - \cdots - e_6$ in $\CP\#6\CPb$ is the Stein filling $S_{1}$. Note that the Stein filling $S_{1}$ can be also be obtained from elliptic fibration on $E(1)$ with three $-1$ sphere sections by removing a tubular neighborhood of a regular torus fiber and these three $-1$ sphere sections. In fact the word $(t_{a}t_{a}t_{b}t_{a})^3 = t_{c_1}t_{c_2}t_{c_3}$ corresponding the star relation in $\Gamma_{1}^{3}$ defines such an elliptic fibration (See Example~\ref{Ex13} for the details). Our second piece, which correponds to the expression $t_{c_1}t_{c_2}t_{c_3}$, is the disk bundle over torus with the Euler number $-3$ in $Y$ is the plumbing $\mathbb{P}_1$. 

On the other hand, the star surgery along $S_1$ is the resolution of the symplectic singularitiy determined by $S_1$. Let us spell out more clearly what we mean by this. The singularity determined by $S_{1}$ is of type $\tilde{E_{6}}$, which is also commonly known as $T_{3,3,3}$ singularity or an ordinary triple point singularity
in various literature. The star surgery on symplectic $4$-manifold $X$ along $S_{1}$ corresponds to a deformation of $S_{1}$ in $X$ and then followed by blow up at the resulting singularity, which is the oridinary triple point $P$ determined by colapsing the vanishing cycles corresponding to $\mathbb{S}_{1}$. Let $T_{P}$ be the exceptional divisor arising from the blow up at point $P$, which is a symplectic torus self intersection $-3$. Thus, we obtain the symplectic $4$-manifold $X_1$ with the symplectic plumbing $\mathbb{P}_1$. Recall the analogy with the rational blowdown operation along $C_2$ plumbing (i.e $- 4$ sphere), which correponds the fiber sum with $\CP$ along a quadric $2H$.

\emph{Case 2}. In the case $g \geq 2$, the reverse GSR surgery being symplectic can be proved using the $3$-fold sum operation of Symington, an adaptation of the symplectic sum operation for positively intersecting surfaces (see Theorem~\ref{3-fold sum} above). For details on the topology of $3$-fold sum and it's application to show the rational blow-downs operation is symplectic, the reader is refered to Symington~\cite{Sym1, Sym2, Sym3}. Consider the plumbing $P_g$ on symplectic $4$-manifold $Y$ for $g \geq 2$. Our sequence of 3-fold symplectic sums will use $Y$ with the plumbing $P_g$, the complex surface $W(g)$ with the genus $g$ symplectic surface of self-intersection $2$, and the sphere with self-intersection $-g$ (see Example~\ref{Ex13}, where the pencil with two base points on $W(g)$ and a sphere $\mathcal{S}_{2}$ with self-intersection $g$ is provided), the Hirzebruch surfaces $\mathbb{F}_{g-1}$, $\cdots$, $\mathbb{F}_{2}$ with the  positive and negative sections $\Delta_{\pm}$ (for $(g-1) \geq i \geq 2$) and the sphere fibers all between them, and $\CP$ with two lines of square $1$. The result of the sequence of $3$-fold sums yields to the symplectic $4$-manifold $Y_{g}$. The GSR surgery on symplectic $4$-manifold $X$ along $S_{g}$ corresponds to collapsing all the vanishing cycles of $S_{g}$ in $X$ and then followed by resolution of the resulting singularity, which is more general $3$-tuple point singularity $P$. The resolution of $P$ introduces the exceptional curves $T_{P}$, $R_{1}$, $\cdots$, $R_{g-1}$, where $T_{P}$ is $-3$ genus $g$ surface, and $R_{1}$, $\cdots$, $R_{g-1}$ are $-2$ spheres resulting from the symplectic resolution. 
\end{proof}

\begin{rmk} The above Theorem can also proved using Theorem~\ref{convex} and the fact that $S_{g}$ is Stein filling of the boundary $3$-manifold. The above operations replace the strong symplectic filling $\mathbb{P}_g$ with a Stein filling $S_{g}$, and vice versa. The (generalized) chain relations also correspond to the symplectic surgery operations, which can be defined along the singularities given by the chain relation and it's generalizations. The surgery operation can be analogously defined as in GSR surgery case using the plumbed $4$-manifolds corresponding to the (generalized) chain relation $(t_{a_1}t_{a_2} \cdots t_{a_m})^{m} = t_{d_1}t_{d_2}$ and $(t_{a_1} \cdots t_{a_m})^{2m+2} = t_{c}$, and their generalizaton. In fact, these surgeries correspond to the Stein fillings $V_{m}$ and $W_{m}$  determined by the left hand sides of these formulas, replaced by the strong symplectic fillings determined by the right had side. Note that the right hand sides correspond to the disk bundles over genus $g$ surfaces with Euler numbers $-1$ and $-2$ respectively. Several examples via the chain relation will be constructed in Section~\ref{words}, which explains the corresponding surgeries. The new symplectic surgery operations based on these ideas and interesting families of Lefschetz fibrations with exotic total spaces will appear in an upcoming work of Akhmedov and Karakurt~\cite{AKa}.
\end{rmk}

\section{Lefschetz fibrations via GSR Surgery and other relations in the mapping class group}\label{words}

In what follows, we will construct many examples of  over $\mathbb{S}^2$ by applying GSR surgery. We also include several Lefschetz fibrations over $\mathbb{S}^2$ obtained via the chain relations in the mapping class group. This is a natural continuation of work begun by the first author in \cite{AP, AMon1, AMon2}, where the lantern and the generalized lantern relations were used. 

In the first few examples, we will use genus one and two Lefschetz fibrations and apply the degenerate cases of star relation. Subsequently, we provide examples using the Lefschetz fibrations of genus $g \geq 3$. Below, we also provide several genus two, genus three and higher genus Lefschetz fibrations via the chain relations. 

\subsection{The star relations on genus one Lefschetz fibrations}

\begin{exmp}\label{ex1} It is well-known that the mapping class group $\Gamma_1$ of the torus is $SL(2, \mathbb{Z})$. It is 
generated by the Dehn twists $t_{a}$ and $t_{b}$ along the standard curves $a$ and $b$ that generate the first homology of the torus and subject to the relations \[ t_{a}t_{b}t_{a}=t_{b}t_{a}t_{b}\qquad {\mbox{ and }} \qquad (t_{a}t_{b})^6=1 \]

Classification of genus $1$ Lefschetz fibrations were given by Moishezon \cite{M}. He showed that after a possible perturbation such a fibration over $\mathbb{S}^2$ is equivalent to one of the fibrations given by the words $(t_{a}t_{b})^{6n}=1$ in $\Gamma_1$. The total space of the Lefschetz fibration corresponding to $(t_{a}t_{b})^{6n}=1$ is $E(n)$, where $E(n)$ is a simply connected elliptic surface without multiple fibers.

Let us consider an elliptic Lefschetz fibration on $E(n)$ with monodromy $(t_{a}t_{b})^{6n}=1$. We can view star relation on torus as ($t_{a}t_{a}t_{b}t_{a})^3 = t_{c_1}t_{c_2}t_{c_3}$, where all $3$ boundary curves $c_1$, $c_2$ and $c_3$ bound disks on torus, thus we are using a degenerate case of star relation. Using the braid relation $t_{a}t_{b}t_{a}=t_{b}t_{a}t_{b}$, let us write the monodromy of $E(n)$ as $1= (t_{a}t_{b})^{6(n-1)}(t_{a}t_{b})^6 = (t_{a}t_{b})^{6(n-1)} (t_{a}t_{a}t_{b}t_{a})^3$ and apply the star relations to the subword $(t_{a}t_{a}t_{b}t_{a})^3$. Notice that the resulting complex surface admits fibration with $6(n-1)$ nodal fibers and one additional singular fiber which consist of $-3$ torus and three disjoint $-1$ spheres intersecting $-3$ torus at three distinct points.  The last singular fiber arises from pinching the three boundary curves $c_1$, $c_2$, and $c_3$, after the GSR surgery using the word  $(t_{a}t_{a}t_{b}t_{a})^3$. It is easy to see that the resulting complex surface has the invariants $c_{1}^2 = -3$ and $\chi_{h} = n-1$, and it is simply connected if $n > 1$ and has the fundamental group isomorphic to $\mathbb{Z} \times \mathbb{Z}$ if $n=1$. By blowing down three $-1$ spheres and using Moishezon's classification of genus $1$ Lefschetz fibrations, we verify that one star relation and blowdown of three $-1$ spheres yields to $E(n-1)$ for $n > 1$ and $\mathbb{T}^2 \times \mathbb{S}^2$ if $n=1$. In fact, the singularity given by $(t_{a}t_{a}t_{b}t_{a})^3$ corresponds to the elliptic triple point singularity of the form $x^{3}+y^{3}+z^{3}=0$ and $-3$ torus is an exceptional torus arising from the resolution. Thus, the star relation and it's corresponding surgery are very useful tools for studying and classifying elliptic Lefschetz fibrations. Notice that the above operation can be reversed: take a regular fiber $F$ of $E(n-1)$ for $n > 1$ or  of $\mathbb{T}^2 \times \mathbb{S}^2$, blow up $F$ at three distinct points, then sum along strict transform of $F$  with $\CP\#6\CPb$ along a torus of square $+3$ given by the class $3H - e_1- \cdots - e_6$ to get $E(n)$. 
\end{exmp}

\subsection{The star and chain relations on genus two Lefschetz fibrations}

The next set of genus two examples illustrates how to apply the degenerate cases of star relation to genus two fibration. These examples also demonstrate the importance of these operations to study the genus two Lefschetz fibrations over $\mathbb{S}^2$. We also include several genus two Lefschetz fibrations over $\mathbb{S}^2$ obtained via the chain relations.

\begin{exmp}\label{ex2} We start with the genus two Lefschetz fibration with the monodromy $(t_{b_1}^{2}t_{b_{2}}t_{b_{3}}t_{b_{4}}t_{b_{5}})^5 = 1$ in the mapping class group $\Gamma_1$. As we remarked before, the total space of this fibration is $Y(2) = K3\#2\CPb$. We will make use of the degenerate case of star relation $(t_{a_0}t_{a_1}t_{a_2} t_{a_{3}})^{3} = t_{c_1} t_{c_2}t_{c_3}$, by assuming that the parallel curves $a_0$ and $a_1$ represent the class of $b_1$ curve, the boundary curves $c_1$ and $c_3$ both represents the class of $b_5$ curve, and $c_2$ curve bounds a disk on the closed genus two surface. Under these assumptions, the star relations $(t_{a_0}t_{a_1}t_{a_2} t_{a_{3}})^{3} = t_{c_1} t_{c_2}t_{c_3}$ on the genus two surface becomes $(t_{b_1}t_{b_1}t_{b_2} t_{b_{3}})^{3} = t_{b_5}t_{b_5}$. 

We apply the conjugations, the braid and the star relations $(t_{b_1}t_{b_1}t_{b_2}t_{b_{3}})^{3} = t_{b_5}t_{b_5}$ to the monodromy $(t_{b_{1}}^{2}t_{b_{2}}t_{b_{3}}t_{b_{4}}t_{b_{5}})^5 = 1$ to construct the genus two Lefschetz fibrations over $\mathbb{S}^2$. By applying the repeated conjugations of the subword $t_{b_{4}}t_{b_{5}}$ in the given monodromy by the subword $t_{b_1}^{2}t_{b_2} t_{b_{3}}$, i.e. using the identity ${t_{b_1}}^{2}t_{b_{2}}t_{b_{3}}t_{b_{4}}t_{b_{5}}{t_{b_1}}^{2}t_{b_{2}}t_{b_{3}} = t_{b_1}^{2} t_{b_2} t_{b_3}(t_{b_4}t_{b_5})(t_{b_1}^{2} t_{b_2} t_{b_3})^{-1}(t_{b_1}^{2} t_{b_2} t_{b_3})^{2}$, we will collect the expressions $(t_{b_1}t_{b_1}t_{b_2} t_{b_{3}})^{5}$ to the end of our monodromy and obtain the following expression for the word $(t_{b_{1}}^{2}t_{b_{2}}t_{b_{3}}t_{b_{4}}t_{b_{5}})^5 = 1$
\begin{align*}
t_{b_1}^{2} t_{b_2} t_{b_3}(t_{b_4}t_{b_5})(t_{b_1}^{2} t_{b_2} t_{b_3})^{-1}(t_{b_1}^{2} t_{b_2} t_{b_3})^{2}(t_{b_4}
t_{b_5})(t_{b_1}^{2} t_{b_2} t_{b_3})^{-2}(t_{b_1}^{2} t_{b_2} t_{b_3})^{3}(t_{b_4}t_{b_5})\\(t_{b_1}^{2} t_{b_2}  t_{b_3})^{-3}(t_{b_1}^{2} t_{b_2} t_{b_3})^{4}(t_{b_4}t_{b_5})(t_{b_1}^{2} t_{b_2} t_{b_3})^{-4}(t_{b_1}^{2} t_{b_2} t_{b_3})^{5}t_{b_4}t_{b_5} = 1
\end{align*}

 We next apply star substitution to the sub-expression  $(t_{b_1}^{2} t_{b_2} t_{b_3})^{3}$ inside of expression $(t_{b_1}^{2} t_{b_2} t_{b_3})^{5}$. From the resulting monodromy, it is easy to see that one obtains genus two  fibration over $\mathbb{S}^2$ with $20$ nodal non-separating singular fibers and one additional fiber which contains $-1$ sphere. The last singular fiber arises from of the boundary curve $c_2$, which according to our set up bounds a disk. By blowing this $-1$ sphere, we obtain Lefschetz fibrations over $\mathbb{S}^2$ with simply connected total space, $20$ nodal non-separating singular fibers, and with the invariants $c_{1}^2 = -4$, $e = 16$, $\sigma = - 12$. The total space of the resulting genus two fibration is diffeomorphic to the rational surface $\CP\#13\CPb$. After blowdown of $-1$ sphere, the resulting singular fiber corresponding to $t_{b_5}t_{b_5}$ has two components: one torus with self-intersection $-2$ and other sphere with self-intersection $-2$, which corresponds to the singularity of type $I_{2,0,0}$ in \cite{NamU}.

It would be interesting to determine if the resulting Lefschetz fibration is holomorphic or not. Observe that the symplectic Kodaira dimension reduces under the above operation, which is the consequence of Usher's and Li-Yau's Theorems in \cite{U, LY} on how the symplectic Kodaira dimension changes under symplectic connected sums and also the fact that the Kodaira dimension is preserved under blowdown of $-1$ spheres.
\end{exmp}

\begin{exmp}\label{ex3} We start with the genus two Lefschetz fibration with the monodromy $(t_{b_1}t_{b_{2}}t_{b_{3}}t_{b_{4}}t_{b_{5}})^6 = 1$ in the mapping class group of genus two surface. The total space of this fibration is also $Y(2) = K3\#2\CPb$. In \cite{AP}, the lantern relations was studied on the given monodromy and consequently an exotic symplectic $4$-manifolds were constructed. 

We will make use of the chain relation $(t_{a_0}t_{a_1})^{6} = t_{c}$ twice on the given monodromy, by assuming that the curves $a_0$ and $a_1$ represents $b_1$ and $b_2$ curves and then by $b_4$ and $b_5$ curves, and the boundary curve $c$ on torus is a seperating curve $c$ on the genus two surface. Under these assumptions, the chain relations becomes $(t_{b_1}t_{b_2})^{6} = t_{c} = (t_{b_4}t_{b_5})^{6}$.

We will apply the conjugations, the braid relations, and the chain relations $(t_{b_1}t_{b_2})^{6} = t_{c} = (t_{b_4}t_{b_5})^{6}$ to the the monodromy $(t_{b_1}t_{b_2}t_{b_3}t_{b_4}t_{b_5})^{6} = 1$ to construct the genus two Lefschetz fibrations over $\mathbb{S}^2$. By applying the braid relations~\ref{com&braid.lem} repeatedly and then followed by the conjugation, we have 
\begin{align*}
1=((t_{b_1}t_{b_2}t_{b_3}t_{b_4}t_{b_5})^{2})^{3} = (t_{b_1}t_{b_2}t_{b_1}t_{b_3}t_{b_2}t_{b_4}t_{b_3}t_{b_5}t_{b_4}t_{b_5})^{3}= \\ ((t_{b_1}t_{b_2})^{2})(t_{b_2}^{-1}t_{b_3}t_{b_2})(t_{b_4}t_{b_3}t_{b_4}^{-1})(t_{b_4}t_{b_5})^{2})^{3}. 
\end{align*}

Next, using the conjugations of the above word by powers of $(t_{b_1}t_{b_2})^{2}$ and $(t_{b_5}t_{b_6})^{2}$, we collect the expressions $(t_{b_1}t_{b_2})^{6}(t_{b_4}t_{b_5})^{6}$ to the front of the word. We have 
\begin{align*}
1 = (t_{b_1}t_{b_2}t_{b_4}t_{b_5})^{6}(t_{b_1}t_{b_2}t_{b_4}t_{b_5})^{-4}(t_{b_2}^{-1}t_{b_3}t_{b_2})_{cj}(t_{b_4}t_{b_3}t_{b_4}^{-1})_{cj}(t_{b_1}t_{b_2}t_{b_4}t_{b_5})^{4}\\(t_{b_1}t_{b_2}t_{b_4}t_{b_5})^{-2}
(t_{b_2}^{-1}t_{b_3}t_{b_2})_{cj}(t_{b_4}t_{b_3}t_{b_4}^{-1})_{cj}(t_{b_1}t_{b_2}t_{b_4}t_{b_5})^{2}(t_{b_2}^{-1}t_{b_3}t_{b_2})_{cj}(t_{b_4}t_{b_3}t_{b_4}^{-1})_{cj}
\end{align*}

Notice that we can apply two chain substitution to the sub-expression $(t_{b_1}t_{b_2})^{6}$ and $(t_{b_4}t_{b_5})^{6}$ inside of expression $(t_{b_1}t_{b_2})^{6}(t_{b_4}t_{b_5})^{6}$. From the resulting monodromy, it is easy to see that after applying first chain relation (say using $(t_{b_1}t_{b_2})^{6}= t_{c}$), we obtain genus two Lefschetz fibration over $\mathbb{S}^2$ with $18$ non-separating singular fibers and one separating singular fiber $c$. The last singular fiber arises from the separating curve $c$. The global monodromy of both Lefschetz fibrations can be written down explicitly from the word above. The total space of this Lefschetz fibration is homemorphic to $\CP\#12\CPb$ \cite{freedman}. Using the symplectic Kodaira dimension, it can be shown that the total space is not diffeomorphic to $\CP\#12\CPb$. After applying the second chain relation $(t_{b_4}t_{b_5})^{6}= t_{c'}$, we obtains genus two Lefschetz fibration over $\mathbb{S}^2$ with $6$ non-separating singular fibers and $2$ separating singular fibers along the parallel curves $c$ and $c'$. The total space of the resulting genus two fibration is diffeomorphic to the ruled surface $\mathbb{T}^2 \times \mathbb{S}^2\#4\CPb$ and the fibration on it is Matsumoto's famous genus two Lefschetz fibration of type $(6,2)$. Notice that the symplectic Kodaira dimension did not increase under the above operation.
\end{exmp}

\begin{exmp}\label{ex4} Let us start with the genus two Lefschetz fibration with the monodromy $(t_{b_1}t_{b_{2}}t_{b_{3}}t_{b_{4}})^{10} = 1$ in the mapping class group of genus two surface. As we remarked before, the total space of this fibration is $Z(2)$ = \emph{Horikawa surface}. We will make use of the degenerate case of star relation $(t_{b_1}t_{b_1}t_{b_2} t_{b_{3}})^{3} = t_{b_5}t_{b_5}$ as in Example~\ref{ex2} above.

We will apply the conjugations, braid relations and the star relation $(t_{b_1}t_{b_1}t_{b_2}t_{b_{3}})^{3} = t_{b_5}t_{b_5}$ to the the monodromy $(t_{b_{1}}t_{b_{2}}t_{b_{3}}t_{b_{4}})^{10} = 1$. 
Using the braid relations repeately, expression $(t_{b_{1}}t_{b_{2}}t_{b_{3}}t_{b_{4}})^{2}$ can be written as follows: $(t_{b_{1}}^{2}t_{b_{2}}t_{b_{3}})t_{b_{1}}t_{b_{2}}t_{b_{4}}t_{b_{3}}$ and then by conjugations with $(t_{b_{1}}^{2}t_{b_{2}}t_{b_{3}})$, we collect the expressions $(t_{b_1}t_{b_1}t_{b_2} t_{b_{3}})^{5}$ to the end of our monodromy and obtain the following expression for the word $(t_{b_{1}}t_{b_{2}}t_{b_{3}}t_{b_{4}})^{10} = 1$
\begin{align*}t_{b_1}^{2} t_{b_2} t_{b_3}(t_{b_1}t_{b_2}t_{b_4}t_{b_3})(t_{b_1}^{2} t_{b_2} t_{b_3})^{-1}\\(t_{b_1}^{2} t_{b_2} t_{b_3})^{2}(t_{b_1}t_{b_2}t_{b_4}t_{b_3})(t_{b_1}^{2} t_{b_2} t_{b_3})^{-2}(t_{b_1}^{2} t_{b_2} t_{b_3})^{3}\\
(t_{b_1}t_{b_2}t_{b_4}t_{b_3})(t_{b_1}^{2} t_{b_2}t_{b_3})^{-3}(t_{b_1}^{2} t_{b_2} t_{b_3})^{4}(t_{b_1}t_{b_2}t_{b_4}t_{b_3})\\(t_{b_1}^{2} t_{b_2} t_{b_3})^{-4}(t_{b_1}^{2} t_{b_2} t_{b_3})^{5}t_{b_1}t_{b_2}t_{b_4}t_{b_3} = 1
\end{align*}

Note that we can apply star substitution to the expression $(t_{b_1}^{2} t_{b_2} t_{b_3})^{3}$ inside of expression $(t_{b_1}^{2} t_{b_2} t_{b_3})^{5}$. From the resulting monodromy, it is easy to see that one obtains genus two  fibration over $\mathbb{S}^2$ with $30$ nodal non-separating singular fibers and one additional fiber which contains $-1$ sphere. The last singular fiber arises from of the boundary curve $c_2$. By blowing this $-1$ sphere, we obtain Lefschetz fibrations over $\mathbb{S}^2$ with simply connected total space, $30$ nodal non-separating singular fibers, and with the invariants $c_{1}^2 = -2$, $e = 26$, $\sigma = - 18$.  The total space of the resulting genus two fibration is $Y(2) = K3\#2\CPb$. Again, observe that the symplectic Kodaira dimension reduces under the above operation, which is the consequence of Usher's and Li-Yau's Theorems in \cite{U, LY}. 
\end{exmp}

\begin{exmp}\label{ex5'} We again start with the genus two Lefschetz fibration with the monodromy $(t_{b_1}t_{b_2}t_{b_3}t_{b_4})^{10} = 1$. We will make use of the chain relation $(t_{a_0}t_{a_1})^{6} = t_{c}$, by assuming that the curves $a_0$ and $a_1$ represents $b_1$ and $b_2$ curves and the boundary curve $c$ on torus is a seperating curve $c$ on the genus two surface. Under these assumptions, the chain relations becomes $(t_{b_1}t_{b_2})^{6} = t_{c}$.

We apply the conjugations, the braid relation and the chain relations $(t_{b_1}t_{b_2})^{6} = t_{c}$ to the monodromy $(t_{b_1}t_{b_2}t_{b_3}t_{b_4})^{10} = 1$ to construct the genus two Lefschetz fibrations. First, by applying the braid relations repeatedly, we have 
\begin{align*}
1= ((t_{b_1}t_{b_2}t_{b_3}t_{b_4})^{2})^{5} = (t_{b_1}t_{b_2}t_{b_1}t_{b_2}t_{b_3}t_{b_2}t_{b_4}t_{b_3})^{5}= ((t_{b_1}t_{b_2})^{2}t_{b_3}t_{b_2}t_{b_4}t_{b_3})^{5}
\end{align*}

Next, using the conjugations of the above word by powers of $(t_{b_1}t_{b_2})^{2}$, we will collect the expressions $(t_{b_1}t_{b_2})^{10}$ to the front of the word. Notice that we apply the chain substitution to 
the sub-expression $(t_{b_1}t_{b_2})^{6}$ inside of expression $(t_{b_1}t_{b_2})^{10}$. From the resulting monodromy, it is easy to see that after applying one chain relation using $(t_{b_1}t_{b_2})^{6}= t_{c}$), we obtain genus two Lefschetz fibration over $\mathbb{S}^2$ with $28$ non-separating singular fibers and one separating singular fiber. The last singular fiber arises from the separating curve $c$. The total space of this Lefschetz fibration is an exotic copy of to $3\CP\#20\CPb$. 

Next, we show how to find two chain relations $(t_{b_1}t_{b_2})^{6} = t_{c}$. The computation is more subtle is this case, and we make use of Lemma~\ref{power}. By setting $m=4$, $k=2$ and $g=2$ in Lemma~\ref{power}, we get the following relation:

\begin{align}
(t_{b_1}t_{b_2}t_{b_3}t_{b_4})^{3} = (t_{b_1}t_{b_2})^{3}(t_{b_3}t_{b_2}t_{b_1})(t_{b_4}t_{b_3}t_{b_2})
\end{align}
which can be conjugated to have the following form

\begin{align*}
((t_{b_1}t_{b_2})^{3}(t_{b_3}t_{b_2}t_{b_1})(t_{b_4}t_{b_3}t_{b_2}))= ((t_{b_1}t_{b_2})^{4}(t_{b_2}t_{b_1}t_{b_2})\\(t_{b_2}t_{b_1}t_{b_2})^{-1}(t_{b_3})t_{b_2}t_{b_1}t_{b_2}{t_{b_2}}^{-1}(t_{b_4}t_{b_3})t_{b_2})
\end{align*}

\noindent Using the global monodromy and conjugating by powers of $((t_{b_1}t_{b_2})^{4}(t_{b_2}t_{b_1}t_{b_2})$, we can collect $(t_{b_1}t_{b_2})^{9}(t_{b_2}t_{b_1}t_{b_2})^{3}$ from the word $(t_{b_1}t_{b_2}t_{b_3}t_{b_4})^{10}= 1$. The last expression is braid equvalent to of $(t_{b_1}t_{b_2})^{12}$. Thus, we can apply at least to chain relations $(t_{b_1}t_{b_2})^{6} = t_{c}$ to $(t_{b_1}t_{b_2}t_{b_3}t_{b_4})^{10} = 1$. We obtain genus two Lefschetz fibration over $\mathbb{S}^2$ with $16$ non-separating singular fibers and two separating singular fibers. The last two singular fiber arise from two chain relations. The total space of this Lefschetz fibration is homeomorphic to $\CP\#11\CPb$ \cite{freedman}.  Using the symplectic Kodaira dimension, it can be shown that the total space is not diffeomorphic to $\CP\#11\CPb$. In fact, we beleive that one additional chain relation can be applied, which would yield to Xiao Gang's beautiful genus two Lefschetz fibration of type $(4,3)$ on ruled surface $\mathbb{T}^2 \times \mathbb{S}^2\#3\CPb$ \cite{Xiao}. 
\end{exmp}

\newpage

\subsection{The chain relations on genus $g\geq 3$ Lefschetz fibrations}

\begin{exmp}\label{ex5} We consider genus three Lefschetz fibration over $\mathbb{S}^2$ with the monodromy $(t_{b_1}t_{b_2} \cdots t_{b_6})^{14} = 1$ in the mapping class group $\Gamma_3$. Recall from Lemma~\ref{E3} that the total space of this fibration is Horikawa surface $Z(3)$. The key idea is to find various chain relations of the form $(t_{a_0}t_{a_1})^{6} = t_{c}$ and $(t_{a_0}t_{a_1}t_{a_2}t_{a_3})^{10} = t_{d}$. Such expressions can easily be found by conjugations or by applying Lemma~\ref{power}. 

Applying the braid relations repeatedly, we have 
\begin{align*}
1= ((t_{b_1}t_{b_2}t_{b_3}t_{b_4}t_{b_5}t_{b_6})^{2})^{7} = (t_{b_1}t_{b_2}t_{b_1}t_{b_3}t_{b_2}t_{b_4}t_{b_3}t_{b_5}t_{b_4}t_{b_6}t_{b_5}t_{b_6})^{7}=\\ ((t_{b_1}t_{b_2})^{2}{t_{b_2}^{-1}t_{b_3}t_{b_2}t_{b_4}t_{b_3}t_{b_5}t_{b_4}t_{b_6}t_{b_5}t_{b_6}})^{7}. 
\end{align*}

Using the conjugations of the above word by appropriate powers of $(t_{b_1}t_{b_2})^{2}$, we will collect the expressions $(t_{b_1}t_{b_2})^{14}$ to the front of the word. Notice that we can apply up at least two chain substitution to the sub-expression $(t_{b_1}t_{b_2})^{6}$ inside of expression $(t_{b_1} t_{b_2})^{14}$. From the resulting monodromy, it is easy to see that after applying two chain relation using $(t_{b_1}t_{b_2})^{6}= t_{c}$, we obtain genus three Lefschetz fibration over $\mathbb{S}^2$ with $60$ non-separating singular fibers and two separating singular fibers.  We can also apply the chain relation $(t_{b_1}t_{b_2}t_{b_3}t_{b_4})^{10} = t_{d}$ to the word $(t_{b_1}t_{b_2} \cdots t_{b_6})^{14} = 1$. The sub expression $(t_{b_1}t_{b_2}t_{b_3}t_{b_4})^{10}$ can be found by using Lemma~\ref{power} or simply by conjugation by powers of $(t_{b_1}t_{b_2}t_{b_3}t_{b_4})$. The resulting genus three Lefschetz fibration over $\mathbb{S}^2$ will have $44$ non-separating singular fibers and one separating singular fiber. Applying one chain relation of type $(t_{b_1}t_{b_2}t_{b_3}t_{b_4})^{10} = t_{d}$ and other chain relation of type $(t_{b_1}t_{b_2})^{6}= t_{c}$), gives genus three Lefschetz fibration over $\mathbb{S}^2$ with $32$ non-separating singular fibers and two separating singular fiber. 

Our computation method applies equally well to $(t_{c_1}t_{c_2} \cdots t_{c_{2g}})^{2(2g+1)n} = 1$ for any $g \geq 3$ and $n \geq 1$ in $\Gamma_{g}$, using the various forms of chain relations. Let us explain the case of $(t_{c_1}t_{c_2} \cdots t_{c_{2g}})^{2(2g+1)} = 1$ for any $g \geq 3$, $n=1$ and leave the details of other cases and applications for a future paper.
Applying the conjugation, we have 
\begin{align*}
(t_{b_1}t_{b_2}t_{b_3}t_{b_4} \cdots t_{b_{2g-1}}t_{b_{2g}}) = (t_{b_1}t_{b_2} \cdots t_{b_{2g-3}}t_{b_{2g-2}})
(t_{b_{2g-1}}t_{b_{2g}})=\\
(t_{b_1}t_{b_2}\cdots t_{b_{2g-2}})(t_{b_{2g-1}}t_{b_{2g}})_{conj}
\end{align*}
Consequently, we have 
\begin{align*}
1= (t_{b_1}t_{b_2}t_{b_3}t_{b_4} \cdots t_{b_{2g-1}}t_{b_{2g}})^{2(2g+1)} = 
(t_{b_1}t_{b_2}\cdots t_{b_{2g-2}})^{2(2g+1)}\\
(\widetilde{t_{1}} \widetilde{t_{2}}\cdots \widetilde{t}_{8g+4})
\end{align*}

\noindent By using the sub expression $(t_{b_1}t_{b_2}\cdots t_{b_{2g-2}})^{4g-2}$, we can apply one chain relation using the formula $(t_{b_1}t_{b_2}\cdots t_{b_{2g-2}})^{4g-2}= t_{c}$. The total space of the resulting genus $g$ Lefschetz fibrations has the type $(16g-4, 1)$.  Many other forms of chain relations can be applied as well.
\end{exmp}

\begin{exmp}\label{ex6} We consider genus three Lefschetz fibration over $\mathbb{S}^2$ with the monodromy $(t_{b_1}t_{b_2} \cdots t_{b_6}t_{b_{7}})^{8} = 1$ in the mapping class group $\Gamma_3$. Recall from Lemma~\ref{E3} that the total space of this fibration is $Y(3)$. Again, the key idea is to apply the chain relations using 
the subsurfaces which are torus or genus two surfaces and make use of the relations $(t_{a_0}t_{a_1})^{6} = t_{c}$ and $(t_{a_0}t_{a_1}t_{a_2}t_{a_3})^{10} = t_{d}$. Such expressions can easily be found by conjugations or by applying Lemma~\ref{power}. Applying the braid relations repeatedly, we have 
\begin{align*}
1= ((t_{b_1}t_{b_2}t_{b_3}t_{b_4}t_{b_5}t_{b_6}b_{7})^{2})^{4} \\
= (t_{b_1}t_{b_2}t_{b_1}t_{b_3}t_{b_2}t_{b_4}t_{b_3}t_{b_5}t_{b_4}t_{b_6}t_{b_5}t_{b7}t_{b_6}t_{7})^{4}= \\
((t_{b_1}t_{b_2})^{2}{t_{b_2}^{-1}t_{b_3}t_{b_2}t_{b_4}t_{b_3}t_{b_5}t_{b_4}t_{b_6}t_{b_5} 
{t_{b_6}}^{-1}(t_{b_6}t_{b_7})^{2}})^{4}
\end{align*}
Using the conjugations of the above word by $(t_{b_1}t_{b_2})^{2}$ and $(t_{b_6}t_{b_7})^{2}$, we will collect the expressions $(t_{b_1}t_{b_2})^{8}(t_{b_6}t_{b_7})^{8}$ to the front of the word. Notice that we can apply up at least two chain substitution to the sub-expression $(t_{b_1}t_{b_2})^{6}$  and $(t_{b_6}t_{b_7})^{6}$ inside of expression $(t_{b_1}t_{b_2})^{8}(t_{b_6}t_{b_7})^{8}$. From the resulting monodromy, it is easy to see that after applying two chain relation using $(t_{b_1}t_{b_2})^{6}= t_{c}$ and $(t_{b_6}t_{b_7})^{6}= t_{d}$, we obtain genus three Lefschetz fibration over $\mathbb{S}^2$ with $32$ non-separating singular fibers and two separating singular fiber along the curves $c$ and $d$. 

Our computation applies equally to $(t_{c_1}t_{c_2} \cdots t_{c_{2g}})^{(2g+2)n} = 1$ for any $g \geq 3$ and $n \geq 1$ and  $(t_{c_1}t_{c_2} \cdots t_{c_{2g-1}}t_{c_{2g}}t_{{c_{2g+1}}}^2t_{c_{2g}}t_{c_{2g-1}}\cdots t_{c_2}t_{c_1})^{2n} = 1$ any $g \geq 2$ and $n \geq 3$ using the various forms of chain relation. We explain the case of $(t_{c_1}t_{c_2} \cdots t_{c_{2g}})^{2g+2} = 1$ for any $g \geq 3$ and $n=1$ and leave the details of other cases and applications for a future paper.
Applying the conjugation, we have 
\begin{align*}
(t_{b_1}t_{b_2}t_{b_3}t_{b_4} \cdots t_{b_{2g+1}}t_{b_{2g+1}}) = (t_{b_1}t_{b_2}\cdots t_{b_g}) t_{b_{g+1}}
(t_{b_{g+2}} \cdots t_{b_{2g}}t_{b_{2g+1}})=  \\
(t_{b_1}t_{b_2}\cdots t_{b_g})(t_{b_{g+2}} \cdots t_{b_{2g}}t_{b_{2g+1}})(t_{b_g})_{conj}
\end{align*}
Consequently, we have 
\begin{align*}
1= (t_{b_1}t_{b_2}t_{b_3}t_{b_4} \cdots t_{b_{2g+1}}t_{b_{2g+1}})^{2g+2} = 
(t_{b_1}t_{b_2}\cdots t_{b_g})^{2(g+1)}(t_{b_{g+2}} \cdots \\ \cdots t_{b_{2g}}t_{b_{2g+1}})^{2(g+1)}
(\widetilde{t_{1}} \widetilde{t_{2}}\cdots \widetilde{t}_{2g+2})
\end{align*}

\noindent By collecting the expression $(t_{b_1}t_{b_2}\cdots t_{b_g})^{2(g+1)}(t_{b_{g+2}} \cdots t_{b_{2g}}t_{b_{2g+1}})^{2(g+1)}$, we can apply one, two, three or four chain relations using the formulas \\
$(t_{b_1}t_{b_2}\cdots t_{b_g})^{g+1}= t_{c}t_{d}$ and $(t_{b_{g+1}}t_{b_2}\cdots t_{b_{2g+1}})^{g+1}= t_{c}t_{d}$. The total space of the resulting genus $g$ fibrations, we call them $X_{g,i}$ for $i=1, 2, 3, 4$, are exotic $4$-manifolds when $i =1, 2, 3$ and the ruled surface when $i=4$. \end{exmp}

\section*{Acknowledgments} We are very grateful to D. Auroux for some helpful discussions we have had back in September 2018 regarding some of the material presented here. We also thank H. Endo and N. Monden, with whom the first author had an email exchange with at the early stages of this project. We also thank C. Karakurt for his interest in our work and many comments which led to the present improved version of the article. Both authors would like to thank Simons Foundation for supporting our projects and the Department of Mathematics at Harvard University, where part of this work was done, for its hospitality. A. Akhmedov was supported by Simons Research Fellowship and  Collaboration Grants for Mathematicians by Simons Foundation. L. Katzarkov was supported by Simons research grant, NSF DMS 150908, ERC Gemis, DMS-1265230, DMS-1201475 OISE-1242272 PASI. Simons collaborative Grant - HMS, Simons investigator grant - HMS. HSE- grant, HMS and automorphic forms. L. Katzarkov was also partially supported by Laboratory of Mirror Symmetry NRU HSE, RF Government grant, ag. No 14.641.31.0001. 

\end{document}